\documentclass{amsart}
\usepackage{amsmath,amssymb,amsthm,comment,todonotes}
\usepackage[pagewise]{lineno}
\theoremstyle{definition}
\newtheorem{definition}{Definition}[section]
\theoremstyle{plain}
\newtheorem{proposition}[definition]{Proposition}
\newtheorem{lemma}[definition]{Lemma}
\newtheorem{theorem}[definition]{Theorem}
\newtheorem{corollary}[definition]{Corollary}
\newtheorem*{theorem*}{Theorem}
\theoremstyle{remark}
\newtheorem{remark}[definition]{Remark}

\title{$G$-deformations of maps into projective space}
\author{Mason Pember}
\address[M.~Pember]{Dipartimento di Matematica, Politecnico di Torino,
Corso Duca degli Abruzzi 24, I-10129 Torino, Italy}
\email{mason.pember@polito.it}


\begin{document}
\maketitle

\begin{abstract}
$G$-deformability of maps into projective space is characterised by the existence of certain Lie algebra valued 1-forms. This characterisation gives a unified way to obtain well known results regarding deformability in different geometries. 
\end{abstract}

\section{Introduction}
It is well known that isothermic surfaces are the only surfaces in conformal geometry that admit non-trivial second order deformations~\cite{C1923} and that $R$- and $R_{0}$-surfaces are the only surfaces in projective geometry that admit non-trivial second order deformations~\cite{C1920ii, F1937}. In \cite{MN2006} it is shown that $\Omega$- and $\Omega_{0}$-surfaces are the only surfaces in Lie sphere geometry that admit non-trivial second order deformations. Motivated by these results we investigate $G$-deformations of smooth maps into $G$-invariant submanifolds of projective space $\mathbb{P}(V)$, where $G$ is a group acting linearly on $V$. This method quickly recovers the aforementioned results regarding deformability in the context of gauge theory.

The examples studied in this paper are all examples of $R$-spaces \cite{T1954}. The author believes that the main theorem of this paper can be used to study deformations in general $R$-spaces and intends to do so in subsequent work. 

It should be noted that Cartan's method of moving frames was utilised in \cite{G1974,J1981} to outline methods for considering deformations of submanifolds of general homogeneous spaces. A different approach is used in this paper that is more suited to recovering gauge-theoretic characterisations of certain classes of surfaces. 

We start by recalling the definition of $k$-th order deformations of maps into homogeneous spaces \cite{G1974,J1981}. Let $N$ be a manifold on which a Lie group $G$, with Lie algebra $\mathfrak{g}$, acts smoothly and let $F:\Sigma\to N$ be a smooth map from a manifold $\Sigma$ into $N$.

\begin{definition}
\label{defn:defo}
Let $k\in\mathbb{N}\cup \{0\}$. We say that $\hat{F}:\Sigma\to N$ is a \textit{$k^{th}$-order $G$-deformation of $F$} if there exists a smooth map $g:\Sigma\to G$ such that for all $p\in\Sigma$ 
\[ g^{-1}(p)\hat{F} \quad \text{and} \quad F\]
agree to order $k$ at $p$. 

If $F$ and $\hat{F}$ are congruent, i.e., $\hat{F}=A F$ for some $A\in G$, we say that the deformation is \textit{trivial}. A map $F:\Sigma\to N$ is said to be \textit{$G$-deformable of order $k$} if it admits a non-trivial $k$-th order $G$-deformation, otherwise $F$ is said to be \textit{$G$-rigid to $k$-th order}. 
\end{definition}

\begin{remark}
\label{rem:chart}
Note that the notion of ``agreeing to order $k$" means that the projections into any chart agree to order $k$. 
\end{remark}

\begin{remark}
\label{rem:transitive}
$k$-th order contact is an equivalence relation. It therefore follows that the notion of $G$-deformability of maps in Definition~\ref{defn:defo} is also an equivalence relation.
\end{remark}

The following lemmata concern the uniqueness and triviality of $G$-deformations and will be used later in this paper. The proofs can be deduced easily using the fact that $G$-deformability is an equivalence relation.  

\begin{lemma}
\label{lem:defouniq}
Let $\hat{F}:\Sigma\to N$ be a $k$-th order $G$-deformation of $F$ via $g:\Sigma\to G$. Then $\hat{F}$ is a $k$-th order $G$-deformation of $F$ via $\tilde{g}:\Sigma\to G$ as well if and only if $F$ is a $k$-th order $G$-deformation of itself via $h:=g^{-1}\tilde{g}$.
\end{lemma}

\begin{lemma}
\label{lem:defotriv}
Suppose that $\hat{F}:\Sigma\to N$ is a $k$-th order $G$-deformation of $F$ via $g:\Sigma\to G$. Then this is a trivial deformation if and only if $g=Ah$ where $A\in G$ and $h:\Sigma\to G$ such that $F$ is a $k$-th order $G$-deformation of itself via  $h:\Sigma\to G$. 
\end{lemma}

Clearly, if $\hat{F}$ is a $k$-th order $G$-deformation of $F$ via $g$ then we may write $\hat{F}=gF$. In this way we may recover $\hat{F}$ from $g$. Furthermore, for any $A\in G$, it is clear that $(Ag)F$ is a $k$-th order deformation of $F$ if and only if $gF$ is a $k$-th order deformation of $F$. This leads us to the following definition:

\begin{definition}
$\eta\in \Omega^{1}(\mathfrak{g})$ is a \textit{$k$-th order infinitesimal deformation of $F$} if $\eta$ satisfies the Maurer-Cartan equation and $gF$ is a $k$-th order $G$-deformation of $F$ for any $g:\Sigma\to G$ satisfying $g^{-1}dg =\eta$. 
\end{definition}

To simplify our exposition in this paper, we shall use the following notation: let $j,k\in\mathbb{Z}$ and define $S_{j,k}:=\{j,...,k\}$ if $j\le k$ and $S_{j,k}:=\emptyset$ if $k<j$.  Let $W$ be a vector bundle over $\Sigma$, suppose that $X_{j},..., X_{k}\in \Gamma T\Sigma$ and let $\sigma\in \Gamma W$. Then for $J\subset S_{j,k}$ with $J= \{j_{1}<...<j_{l}\}$ we let 
\[d_{X_{J}}\sigma:= d_{X_{j_{1}}}(d_{X_{j_{2}}}...(d_{X_{j_{l}}}\sigma)),\]
and
\[ d_{X_{\emptyset}}\sigma := \sigma.\]
Note that we will also use the notation $d_{X_{j_{1}}...X_{j_{l}}}\sigma$ to represent $d_{X_{j_{1}}}(d_{X_{j_{2}}}...(d_{X_{j_{l}}}\sigma))$ when it is convenient. 
We will repeatedly use the Leibniz rule, i.e., if $\sigma,\xi\in \Gamma W$ and $J\subset S_{j,k}$, then 
\[ d_{X_{J}}(\sigma\otimes\xi) = \sum_{K\subset J} (d_{X_{K}}\sigma)\otimes(d_{X_{J\backslash K}}\xi).\]

The focus of this paper is deformability of maps into submanifolds of projective space. Suppose that $V$ is a vector space with projectivisation $\mathbb{P}(V)$ and suppose that $G$ is a Lie group acting linearly on $V$. By Remark~\ref{rem:chart}, $k$-th order contact of immersions into projective space can be determined via sections of these immersions: 

\begin{proposition}
\label{prop:sections}
$\phi,\hat{\phi}:\Sigma\to \mathbb{P}(V)$ agree to order $k$ at $p\in\Sigma$ if and only if for any $v_{0}\in V^{*}$, the sections $\sigma,\hat{\sigma}$ of $\phi$ and $\hat{\phi}$, respectively, such that 
\[v_{0}(\sigma)= v_{0}(\hat{\sigma})=1\]
agree to order $k$ at $p$ on the open set where they are defined. 
\end{proposition}

Now suppose that $S$ is a $G$-invariant submanifold of $\mathbb{P}(V)$. The main theorem of this paper, proved in Section~\ref{sec:defo}, is the following: 

\begin{theorem*}
$\eta\in \Omega^{1}(\mathfrak{g})$ is a $k$-th order infinitesimal deformation of $F:\Sigma\to S$ if and only if $\eta$ satisfies the Maurer-Cartan equation and
\begin{equation}
\label{eqn:nicedefo}
\eta(Y)F\le F,\quad (d_{X_{1}}\eta(Y))F\le F, \quad ...\quad, (d_{X_{1}...X_{k-1}}\eta(Y)) F\le F,
\end{equation}
for all $Y, X_{1},...,X_{k-1},\in \Gamma T\Sigma$. 
\end{theorem*}

The remaining sections of this paper are devoted to using the main theorem to quickly recover results regarding deformability in specific cases. In all of these cases we shall see that rigidity occurs at third order. 

In Section~\ref{sec:proj3} we examine deformability of surfaces in projective 3-space. We obtain a gauge theoretic characterisation of second order deformability via the existence of a certain closed $\mathfrak{sl}(4)$-valued 1-form. 

In Section~\ref{sec:confdefo} we investigate the deformability of hypersurfaces in the conformal $n$-sphere. We recover the result that the only second order deformable surfaces in this case are isothermic surfaces in the conformal $3$-sphere. 

In Section~\ref{sec:legdefo} we investigate deformability of Legendre maps in Lie sphere geometry and projective geometry. We utilise the $\mathbb{R}^{s,t}$ models for these geometries, where $(s,t)=(4,2)$ in Lie sphere geometry and $(s,t)=(3,3)$ in projective geometry. We characterise second order deformability in these cases by the existence of a certain closed 1-form taking values in $\mathfrak{o}(s,t)$. In the case of Lie sphere geometry we see that this characterisation coincides with the gauge theoretic definitions of $\Omega$- and $\Omega_{0}$-surfaces. 

In Section~\ref{sec:projrev} we recover a result of Fubini~\cite{F1916} that relates the second order deformability of a surface in projective 3-space with the second order deformability of its contact lift. This allows us to equate second order deformability in projective 3-space with the gauge theoretic definitions of $R$- and $R_{0}$-surfaces. 

\begin{remark}
In all of the cases considered in this paper, second order deformations come in 1-parameter families. These families arise in these cases because we may multiply the infinitesimal deformation $\eta$ by a constant to obtain a new infinitesimal deformation. 
\end{remark}

\section{Deformations in projective space}
\label{sec:defo} 
Suppose that $V$ is a vector space with projectivisation $\mathbb{P}(V)$ and suppose that $G$ is a Lie group acting linearly on $V$. Let $S$ be a $G$-invariant submanifold of $\mathbb{P}(V)$. Since $k$-th order contact of two maps in $S$ is equivalent to $k$-th order contact as maps into $\mathbb{P}(V)$, we may use Proposition~\ref{prop:sections} to study contact in $S$. Let $F:\Sigma\to S$ be a smooth map from a manifold $\Sigma$ into $S$. The following lemma allows us to characterise deformability of a map $gF:\Sigma\to S$ in terms of the Maurer-Cartan form of $g:\Sigma\to G$:
\begin{lemma}
\label{lem:indproof}
Let $k\in\mathbb{N}$ and suppose that $gF$ is a $(k-1)$-th order $G$-deformation of $F$. Then $F$ and $g^{-1}(p)g F$ agree to order $k$ at $p\in \Sigma$ if and only if for any $v_{0}\in V^{*}$ and $Y,X_{1},...,X_{k-1}\in \Gamma T\Sigma$, 
\[ \theta(Y) d_{X_{S_{1,k-1}}}\sigma = \sum_{K\subset S_{1,k-1}}v_{0}(\theta(Y)d_{X_{K}}\sigma)d_{X_{S_{1,k-1}\backslash K}}\sigma,\]
at $p$, where $\theta=g^{-1}dg$ and $\sigma\in \Gamma F$ such that $v_{0}(\sigma)=1$.
\end{lemma}
\begin{proof}
We shall use strong induction on $k$. Consider the case $k=1$: $F$ and $g^{-1}(p)gF$ agree to order $1$ at $p$ if and only if for any $v_{0}\in V^{*}$, $v_{0}(g^{-1}(p)g\sigma)\sigma$ and $g^{-1}(p)g\sigma$ agree to order $1$ at $p$ where $\sigma\in \Gamma F$ such that $v_{0}(\sigma)=1$. This holds if and only if for any $Y\in T_{p}\Sigma$, 
\[ g^{-1}(p)d_{Y}(g\sigma) = d_{Y}(v_{0}(g^{-1}(p)g\sigma)\sigma).\]
Now using the Leibniz rule and that $\theta_{p}(Y) = g^{-1}(p)d_{Y}g$, this holds if and only if 
\[ \theta_{p}(Y)\sigma + d_{Y}\sigma = v_{0}(\theta_{p}(Y)\sigma)\sigma + d_{Y}\sigma.\]
Noting that $d_{\emptyset}\sigma=\sigma$, we see that the proposition holds when $k=1$. 

Let $n\in\mathbb{N}$ and assume that the proposition holds for all $k< n$ and assume that $F$ and $\hat{F}$ are $(n-1)$-th order deformations of each other. Let $Y,X_{1},...,X_{n-1}\in \Gamma T\Sigma$. Then for any $K\subset \{1,...,n-1\}$ with $|K|<n-1$ we have, by our inductive hypothesis, 
\begin{equation}
\label{eqn:indhyp}
\theta(Y)d_{X_{K}}\sigma = \sum_{L\subset K}v_{0}(\theta(Y)d_{X_{L}}\sigma)d_{X_{K\backslash L}}\sigma.
\end{equation}

Since $F$ and $\hat{F}$ are $(n-1)$-th order deformations of each other we have that for any $v_{0}\in V^{*}$ and $X_{1},..., X_{n-1}\in \Gamma T\Sigma$, 
\[ g^{-1}d_{X_{S_{1,n-1}}}g\sigma - \sum_{K\subset S_{1,n-1}}v_{0}(g^{-1}d_{X_{K}}g\sigma)d_{X_{S_{1,n-1}\backslash K}}\sigma=0, \]
where $\sigma\in\Gamma f$ such that $v_{0}(\sigma)=1$. Differentiating at $p$ with respect to $X_{0}\in \Gamma T\Sigma$ we get, using the Leibniz rule and that $d_{Y}g^{-1}= - \theta(Y)g^{-1}$, 
\begin{align*}
0&=-\theta_{p}(X_{0})g^{-1}(p)d_{X_{S_{1,n-1}}}g\sigma + g^{-1}(p)d_{X_{0}}d_{X_{S_{1,n-1}}}g\sigma\\
&+ \sum_{K\subset S_{1,n-1}}[v_{0}(\theta_{p}(X_{0})g^{-1}(p)d_{X_{K}}g\sigma)d_{X_{S_{1,n-1}\backslash K}}\sigma\\ 
&-
v_{0}(g^{-1}(p)d_{X_{0}X_{K}}g\sigma)d_{X_{S_{1,n-1}\backslash K}}\sigma - v_{0}(g^{-1}(p)d_{X_{K}}g\sigma)d_{X_{0}X_{S_{1,n-1}\backslash K}}\sigma]\\
&= -\theta_{p}(X_{0})g^{-1}(p)d_{X_{S_{1,n-1}}}g\sigma +d_{X_{S_{0,n-1}}}(g^{-1}(p)g\sigma) \\
&+ \sum_{K\subset S_{1,n-1}}v_{0}(\theta_{p}(X_{0})g^{-1}(p)d_{X_{K}}g\sigma)d_{X_{S_{1,n-1}\backslash K}}\sigma - d_{X_{S_{0,n-1}}}(v_{0}(g^{-1}(p)g\sigma)\sigma).
\end{align*}
Thus, $v_{0}(g^{-1}(p)g\sigma)\sigma$ and $g^{-1}(p)g\sigma$ agree to order $n$ at $p$ if and only if 
\begin{align}
\label{eqn:indproof}
\theta_{p}(X_{0})g^{-1}(p)d_{X_{S_{1,n-1}}}g\sigma =\sum_{K\subset S_{1,n-1}}v_{0}(\theta_{p}(X_{0})g^{-1}(p)d_{X_{K}}g\sigma)d_{X_{S_{1,n-1}\backslash K}}\sigma.
\end{align}
Now, $v_{0}(g^{-1}(p)g\sigma)\sigma$ and $g^{-1}(p)g\sigma$ agree up to order $n-1$ at $p$, thus for any $K\subset S_{1,n-1}$, 
\[ g^{-1}(p)d_{X_{K}}g\sigma  = d_{X_{K}}(v_{0}(g^{-1}(p)g\sigma)\sigma) = \sum_{L\subset K}v_{0}(g^{-1}(p)d_{X_{L}}g\sigma)d_{X_{K\backslash L}}\sigma.\]
Thus, (\ref{eqn:indproof}) becomes
\begin{align*}
0&= -\theta_{p}(X_{0})\sum_{K\subset S_{1,n-1}}v_{0}(g^{-1}(p) d_{X_{K}}g\sigma)d_{X_{S_{1,n-1}\backslash K}}\sigma \\
&+\sum_{K\subset S_{1,n-1}}\sum_{L\subset K}v_{0}(\theta_{p}(X_{0})v_{0}(g^{-1}(p)d_{X_{L}}g\sigma)d_{X_{K\backslash L}}\sigma)d_{X_{S_{1,n-1}\backslash K}}\sigma\\
&= -\sum_{K\subset S_{1,n-1}}v_{0}(g^{-1}(p) d_{X_{K}}g\sigma) \theta_{p}(X_{0})d_{X_{S_{1,n-1}\backslash K}}\sigma \\
&+ \sum_{K\subset S_{1,n-1}}\sum_{L\subset K}v_{0}(g^{-1}(p)d_{X_{L}}g\sigma)v_{0}(\theta_{p}(X_{0})d_{X_{K \backslash L}}\sigma)d_{X_{S_{1,n-1}\backslash K}}\sigma.
\end{align*}
After relabelling we have that
\begin{align*}
0 &= \sum_{K\subset S_{1,n-1}}v_{0}(g^{-1}(p)d_{X_{K}}g\sigma)(-\theta_{p}(X_{0})d_{X_{S_{1,n-1}\backslash K}}\sigma\\
&+\sum_{L\subset (S_{1,n-1}\backslash K)} v_{0}(\theta_{p}(X_{0})d_{X_{L}}\sigma)d_{X_{(S_{1,n-1}\backslash K)\backslash L}}\sigma).
\end{align*}
Using the inductive hypothesis (\ref{eqn:indhyp}) we then have
\[ 0 = - \theta_{p}(X_{0})d_{X_{S_{1,n-1}}}\sigma + \sum_{K\subset S_{1,n-1}}v_{0}(\theta_{p}(X_{0})d_{X_{K}}\sigma)d_{X_{S_{1,n-1}\backslash K}}\sigma.\]
Hence, the result holds for the case $k=n$. Therefore, by induction the result is proved.
\end{proof}

Applying Lemma~\ref{lem:indproof} recursively, one obtains the following theorem:

\begin{theorem}
\label{thm:defozero}
$\eta\in \Omega^{1}(\mathfrak{g})$ is a $k$-th order infinitesimal deformation of $F$ if and only if $\eta$ satisfies the Maurer Cartan equation and for all $r\in \{0,..., k-1\}$, $v_{0}\in V^{*}$ and $Y,X_{1},...,X_{r}\in \Gamma T\Sigma$, 
\[ \eta(Y) d_{X_{S_{1,r}}}\sigma = \sum_{K\subset S_{1,r}}v_{0}(\eta(Y)d_{X_{K}}\sigma)d_{X_{S_{1,r}\backslash K}}\sigma,\]
where $\sigma\in \Gamma F$ such that $v_{0}(\sigma)=1$.
\end{theorem}

We now wish to find an invariant characterisation of deformability in terms of the Maurer-Cartan form, i.e., a characterisation that does not require charts. Essentially this is achieved by taking the characterisation of Theorem~\ref{thm:defozero} and successively applying the Leibniz rule. Let $r\in\{0,..., k-1\}$, $Y,X_{1},...,X_{r}\in\Gamma T\Sigma$ and $v_{0}\in V^{*}$. For $I,J\subset \{1,...,r\}$, contemplate the following equation:
\begin{eqnarray}
\label{eqn:defo}
(d_{X_{I}}\eta(Y))d_{X_{J}}\sigma = \sum_{K\subset J} v_{0}((d_{X_{I}}\eta(Y))d_{X_{K}}\sigma )d_{X_{J\backslash K}}\sigma,
\end{eqnarray}
where $\sigma\in\Gamma F$ such that $v_{0}(\sigma)=1$. 

\begin{lemma}
Suppose that for all $I,J\subset \{1,..., r\}$ with $|I|+|J| < r$, (\ref{eqn:defo}) holds. Then~(\ref{eqn:defo}) holds for all $I,J\subset\{1,...,r\}$ with $|I|=i\in\{0,,...,r\}$ and $|I|+|J|=r$ if and only if~(\ref{eqn:defo}) holds for all $I,J\subset\{1,...,r\}$ with $|I|=i+1$ and $|I|+|J|=r$ . 
\end{lemma}
\begin{proof}
Suppose that~(\ref{eqn:defo}) holds for all $I,J\subset\{1,...,r\}$ with $|I|=i\in\{0,,...,r\}$ and $|I|+|J|=r$. Let $I,J\subset \{1,...,r\}$ with $|I|=i+1$ and $|I|+|J|=r$. Without loss of generality, assume that $\min I< \min J$. Let $a$ denote the smallest element of $I$ and $\hat{I}:= I\backslash\{a\}$. Then by our assumption 
\[ (d_{X_{\hat{I}}}\eta(Y))d_{X_{J}}\sigma = \sum_{K\subset J} v_{0}((d_{X_{\hat{I}}}\eta(Y))d_{X_{K}}\sigma )d_{X_{J\backslash K}}\sigma.\]
Differentiating this with respect to $X_{a}$ at $p$ and using the Leibniz rule we have that
\begin{align*}
&(d_{X_{I}}\eta(Y))d_{X_{J}}\sigma + (d_{X_{\hat{I}}}\eta(Y))d_{X_{\{a\}\cup J}}\sigma \\
=&  \sum_{K\subset J} (v_{0}((d_{X_{I}}\eta(Y))d_{X_{K}}\sigma )d_{X_{J\backslash K}}\sigma  +  v_{0}((d_{X_{\hat{I}}}\eta(Y))d_{X_{\{a\}\cup K}}\sigma )d_{X_{J\backslash K}}\sigma\\
+&  \sum_{K\subset J} v_{0}((d_{X_{\hat{I}}}\eta(Y))d_{X_{K}}\sigma )d_{X_{\{a\}\cup J\backslash K}}\sigma)\\
=&\sum_{K\subset J} v_{0}((d_{X_{I}}\eta(Y))d_{X_{K}}\sigma )d_{X_{J\backslash K}}\sigma  + \sum_{L\subset \{a\}\cup J} v_{0}((d_{X_{\hat{I}}}\eta(Y))d_{X_{L}}\sigma )d_{X_{\{a\}\cup J\backslash L}}\sigma.
\end{align*}
By our supposition, 
\[ (d_{X_{\hat{I}}}\eta(Y))d_{X_{\{a\}\cup J}}\sigma= \sum_{L\subset \{a\}\cup J} v_{0}((d_{X_{\hat{I}}}\eta(Y))d_{X_{L}}\sigma )d_{X_{\{a\}\cup J\backslash L}}\sigma.\]
Thus, 
\[ (d_{X_{I}}\eta(Y))d_{X_{J}}\sigma = \sum_{K\subset J} (v_{0}((d_{X_{I}}\theta(Y))d_{X_{K}}\sigma )d_{X_{J\backslash K}}\sigma.\]

A similar argument can be used to prove the converse. 
\end{proof}

\begin{corollary}
\label{cor:diffset}
Suppose that for all $I,J\subset \{1,..., r\}$ with $|I|+|J| < r$, (\ref{eqn:defo}) holds. Then if~(\ref{eqn:defo}) holds for all $I,J\subset\{1,...,r\}$ with $|I|=i\in\{0,,...,r\}$ and $|I|+|J|=r$, then~(\ref{eqn:defo}) holds for all $I,J\subset\{1,...,r\}$ with $|I|+|J|=r$. 
\end{corollary}

We are now in a position to state the following invariant version of Theorem~\ref{thm:defozero}:

\begin{theorem}
\label{thm:defo}
$\eta\in \Omega^{1}(\mathfrak{g})$ is a $k$-th order infinitesimal deformation of $F:\Sigma\to S$ if and only if $\eta$ satisfies the Maurer-Cartan equation and
\begin{equation}
\label{eqn:nicedefo}
\eta(Y)F\le F,\quad (d_{X_{1}}\eta(Y))F\le F, \quad ...\quad, (d_{X_{1}...X_{k-1}}\eta(Y)) F\le F,
\end{equation}
for all $Y, X_{1},...,X_{k-1},\in \Gamma T\Sigma$. 
\end{theorem}

\begin{proof}
Firstly, notice that~(\ref{eqn:nicedefo}) is equivalent to~(\ref{eqn:defo}) with $|I|=r\in\{0,...,k-1\}$ and $|J|=0$, for any choice of $v_{0}\in V^{*}$. 

Suppose that $\eta$ is a $k$-th order infinitesimal deformation of $F$. Then by Theorem~\ref{thm:defozero}, for any $r\in\{0,...,k-1\}$, $Y, X_{1},..., X_{r}\in\Gamma T\Sigma$ and $v_{0}\in V^{*}$, we have that~(\ref{eqn:defo}) holds for all $I,J\subset\{1,...,r\}$ with $|I|=0$ and $|J|=r$. By Corollary~\ref{cor:diffset} it then follows that~(\ref{eqn:defo}) holds for all $I,J\subset\{1,...,r\}$ with $|I|=r$ and $|J|=0$. 

Conversely, suppose that $\eta$ satisfies the Maurer-Cartan equation and, for any $r\in\{0,...,k-1\}$, $Y, X_{1},..., X_{r}\in\Gamma T\Sigma$ and $v_{0}\in V^{*}$,~(\ref{eqn:defo}) holds for all $I,J\subset\{1,...,r\}$ with $|I|=r$ and $|J|=0$. Then by Corollary~\ref{cor:diffset},~(\ref{eqn:defo}) holds for all $I,J\subset\{1,...,r\}$ with $|I|=0$ and $|J|=r$. By Theorem~\ref{thm:defozero} it then follows that $\eta$ is a $k$-th order infinitesimal deformation of $F$. 
\end{proof}

\section{Projective $3$-space}
\label{sec:proj3}
Cartan~\cite{C1920ii} investigated projective deformability and rigidity of surfaces in projective $3$-space. Modern references on this topic include~\cite{AG1993,F1937,GH1979,JM1994}. It was shown in~\cite{F1937} that the class of second order deformable surfaces in projective $3$-space can be split naturally into two subclasses: $R$- and $R_{0}$-surfaces. A modern account of this can be found in~\cite{F2000i} and a gauge theoretic approach for these surfaces was developed in~\cite{C2012i}. In this section we will use the results from Section~\ref{sec:defo} to study these notions.

Consider projective 3-space $\mathbb{P}(\mathbb{R}^{4})$ with transformation group $\textrm{SL}(4)$. Suppose that $\Sigma$ is a 2-dimensional manifold and let $F:\Sigma\to \mathbb{P}(\mathbb{R}^{4})$ be a smooth map. We can view $F$ as a rank 1 subbundle of the trivial bundle $\underline{\mathbb{R}}^{4}:=\Sigma\times \mathbb{R}^{4}$. Let $F^{(1)}$ denote derived bundle of $F$, i.e., the set of sections of $F$ and derivatives of sections of $F$. Assuming that $F$ is an immersion is equivalent to assuming that $F^{(1)}$ is a rank 3 subbundle of the trivial bundle. Let $T_{1}, T_{2}$ denote the (possibly complex conjugate) asymptotic directions of $F$, i.e., for any $X\in \Gamma T_{1}$, $Y\in \Gamma T_{2}$ and $\sigma\in \Gamma F$, 
\[ d_{X}d_{X}\sigma, d_{Y}d_{Y}\sigma \in \Gamma F^{(1)}.\]
We will make the further assumption that the derived bundle $F^{(2)}$ of $F^{(1)}$ satisfies $F^{(2)} = \underline{\mathbb{R}}^{4}$. In other words, for nowhere zero sections $X\in \Gamma T_{1}$, $Y\in \Gamma T_{2}$ and $\sigma\in \Gamma F$, $d_{X}d_{Y}\sigma$ never belongs to $F^{(1)}$.

\subsection{Second order deformations}
\label{subsec:projsord}
We will now investigate when $F$ admits non-trivial second order $\textrm{SL}(4)$-deformations. By Theorem~\ref{thm:defozero}, $\eta\in \Omega^{1}(\underline{\mathfrak{sl}(4)})$ is a second order infinitesimal deformation of $F$ if and only if $\eta$ satisfies the Maurer-Cartan equation and for all $v_{0}\in (\mathbb{R}^{4})^{*}$ and $X,Y\in\Gamma T\Sigma$
\begin{equation}
\label{eqn:projsord}
\eta \sigma = v_{0}(\eta \sigma)\sigma
\end{equation}
and 
\begin{equation}
\label{eqn:projsord2}
\eta(X)d_{Y}\sigma = v_{0}(\eta(X)\sigma) d_{Y}\sigma + v_{0}(\eta(X)d_{Y}\sigma)\sigma, 
\end{equation}
where $\sigma \in \Gamma F$ such that $v_{0}(\sigma)=1$. 

Suppose that $\eta$ is such a second order infinitesimal deformation. Let $X\in \Gamma T_{1}$ and $Y\in \Gamma T_{2}$. By equation~(\ref{eqn:projsord2}) we have that 
\[ \eta(X)d_{X}\sigma = v_{0}(\eta(X)\sigma) d_{X}\sigma + v_{0}(\eta(X)d_{X}\sigma)\sigma.\]
Differentiating this in the $Y$ direction gives 
\begin{align*}  
(d_{Y}\eta(X))d_{X}\sigma + \eta(X)d_{YX}\sigma &= d_{Y}(v_{0}(\eta(X)\sigma)) d_{X}\sigma  + v_{0}(\eta(X)\sigma) d_{YX}\sigma \\
&+ d_{Y}(v_{0}(\eta(X)d_{X}\sigma))\sigma+ v_{0}(\eta(X)d_{X}\sigma)d_{Y}\sigma.
\end{align*}
Since $\eta$ satisfies the Maurer-Cartan equation, one deduces that the left hand side of this equation is
\[
\eta(X)d_{YX}\sigma \, mod\, F^{(1)}.
\]
Whereas the right hand side is 
\[ v_{0}(\eta(X)\sigma)d_{YX}\sigma\, mod\, F^{(1)}.\]
Similarly, one can show that 
\[ \eta(Y)d_{YX}\sigma = v_{0}(\eta(Y)\sigma)d_{YX}\sigma\, mod\, F^{(1)}.\]
Therefore, in terms of the basis $\{ \sigma, d_{X}\sigma, d_{Y}\sigma, d_{YX}\sigma\}$ for $\underline{\mathbb{R}}^{4}$, $\eta(Y)$ is represented by an upper triangular matrix with $v_{0}(\eta(Y)\sigma)$ on the diagonal. Now since $\eta$ takes values in $\mathfrak{sl}(4)$ and is thus trace free, we must have that $v_{0}(\eta \sigma)=0$. Therefore, 
\[ \eta F =0 \quad \text{and}\quad \eta F^{(1)}\le \Omega^{1}(F).\]
Conversely if $\eta$ satisfies   
\[ \eta F =0 \quad \text{and}\quad \eta F^{(1)}\le \Omega^{1}(F) \]
then clearly (\ref{eqn:projsord}) and (\ref{eqn:projsord2}) hold and thus $\eta$ is a second order infinitesimal deformation of $F$. 

One can show (see \cite[Lemma 3.21]{P2015}) that an $\eta \in \Omega^{1}(\underline{\mathfrak{sl}(4)})$ of the above form satisfies the Maurer-Cartan equation if and only if $\eta$ is closed. Thus, we have arrived at the following proposition:

\begin{proposition}
\label{prop:projsord}
$\eta \in \Omega^{1}(\underline{\mathfrak{sl}(4)})$ is a second order infinitesimal deformation of $F$ if and only if $\eta$ is closed and satisfies
$\eta F =0$ and $\eta F^{(1)}\le \Omega^{1}(F)$.

\end{proposition}

We will now investigate the uniqueness and triviality of second order deformations. According to Lemma~\ref{lem:defouniq} and Lemma~\ref{lem:defotriv}, this is determined by maps $h:\Sigma\to \textrm{SL}(4)$ such that $F$ is a second order deformation of itself via $h$. By Proposition~\ref{prop:projsord}, such a $h$ satisfies 
\begin{equation}
\label{eqn:projuniq}
hF=F, \quad \theta_{h}F =0 \quad \text{and} \quad \theta_{h}F^{(1)}\le \Omega^{1}(F),
\end{equation}
where $\theta_{h}:=h^{-1}dh$. Now $hF=F$ implies that for any $\sigma\in\Gamma F$, 
\[ h\sigma =\lambda \sigma\]
for a smooth function $\lambda$. 
Thus, for any $X\in\Gamma T\Sigma$
\[ (d_{X}h)\sigma + h d_{X}\sigma =\lambda d_{X}\sigma + (d_{X}\lambda)\sigma.\]
Using that $\theta_{h}F=0$ 
\[ h d_{X}\sigma =\lambda d_{X}\sigma + (d_{X}\lambda)\sigma.\]
Differentiating this condition with respect to $Y\in\Gamma T\Sigma$ we have that
\begin{align*} 
h d_{YX}\sigma &= \lambda d_{YX}\sigma + (d_{Y}\lambda)d_{X}\sigma + (d_{X}\lambda)d_{Y}\sigma + (d_{YX}\lambda) \sigma - (d_{Y}h)d_{X}\sigma\\
&= \lambda d_{YX}\sigma \, mod\, F^{(1)}, 
\end{align*}
since $\theta_{h}F^{(1)}\le \Omega^{1}(F)$. Therefore, in terms of the basis $\{\sigma, d_{X}\sigma, d_{Y}\sigma, d_{YX}\sigma\}$, $h$ is represented by a upper triangular matrix with $\lambda$ on the diagonal. Then, since $h$ takes values in $\textrm{SL}(4)$, we must have that $\lambda= \pm 1$ and thus
\[ h|_{F^{(1)}}=\pm id|_{F^{(1)}} \quad \text{and}\quad h|_{\underline{\mathbb{R}}^{4}/F}=\pm id|_{\underline{\mathbb{R}}^{4}/F}.\]
We may then write 
\[ h = \pm (id + \xi),\]
where $\xi$ satisfies $\xi|_{F^{(1)}} = 0$ and $im\xi \le F$. Clearly $\xi$ is trace-free, so $\xi \in\Gamma \underline{\mathfrak{sl}(4)}$. Hence, $h = \pm \exp(\xi)$. Conversely, given an $h$ of such a form, one can easily check that~(\ref{eqn:projuniq}) is satisfied. Thus we obtain the following lemmata:

\begin{lemma}
Suppose that $\hat{F}$ is a second order deformation of $F$ via $g:\Sigma\to G$. Then $\hat{F}$ is a second order deformation of $F$ via $\tilde{g}:\Sigma\to G$ if and only if $\tilde{g} = \pm g\exp(\xi)$ for some $\xi \in\Gamma \underline{\mathfrak{sl}(4)}$ satisfying $\xi |_{F^{(1)}}=0$ and $im\xi \le F$.  
\end{lemma}

\begin{lemma}
$\eta$ is a trivial second order infinitesimal deformation of $F$ if and only if $\eta =d\xi$, where $\xi \in\Gamma \underline{\mathfrak{sl}(4)}$ satisfying $\xi |_{F^{(1)}}=0$ and $im\xi \le F$. 
\end{lemma}

We have therefore proved the main theorem of this subsection:

\begin{theorem}
\label{thm:projsord2}
$F:\Sigma\to \mathbb{P}(\mathbb{R}^{4})$ is deformable of order two if and only if there exists $\eta \in \Omega^{1}(\mathfrak{sl}(4))$, such that $\eta$ is closed,
\[ \eta F =0, \quad \eta F^{(1)}\le \Omega^{1}(F)\]
and $\eta\neq d\xi$ for any $\xi \in\Gamma \underline{\mathfrak{sl}(4)}$ satisfying $\xi |_{F^{(1)}}=0$ and $im\xi \le F$. 
\end{theorem}

In Section~\ref{sec:projrev} we shall see that the deformability of a map into $\mathbb{P}(\mathbb{R}^{4})$ coincides with deformability of its contact lift. In that setting the triviality of deformations can be identified by the vanishing of a certain two-tensor. After this equivalence is established, we see that second order deformability coincides with the gauge theoretic definition of \textit{projectively applicability} given in~\cite{C2012i}.

\subsection{Third order deformations}
We shall now show that rigidity occurs at third order in projective 3-space. Suppose that $\eta$ is a third order infinitesimal deformation of $F$. Then by Theorem~\ref{thm:projsord2}, $\eta$ is closed and satisfies
\[ \eta F = 0 \quad \text{and}\quad \eta F^{(1)}\le \Omega^{1}(F).\]
Furthermore, by Theorem~\ref{thm:defozero}, for any $v_{0}\in (\mathbb{R}^{4})^{*}$ and $X,Y,Z\in \Gamma T\Sigma$, 
\begin{align*}
\eta(X)d_{YZ}\sigma &= v_{0}(\eta(X)d_{YZ}\sigma)\sigma + v_{0}(\eta(X)d_{Y}\sigma)d_{Z}\sigma \\
&+ v_{0}(\eta(X)d_{Z}\sigma)d_{Y}\sigma + v_{0}(\eta(X)\sigma)d_{YZ}\sigma,
\end{align*}
where $\sigma\in\Gamma f$ such that $v_{0}(\sigma)=1$. Now suppose that $Y$ is an asymptotic direction of $F$ and $Z=Y$. Then $d_{YZ}\sigma\in \Gamma F^{(1)}$ and thus $\eta(X)d_{YZ}\sigma\in \Gamma F$. Hence, $v_{0}(\eta(X)d_{Y}\sigma)=0$. Therefore, $\eta F^{(1)} =0$. We will now use that $\eta$ is closed to show that $\eta=0$:
suppose that $X,Y,Z\in\Gamma T\Sigma$. Then, as $\eta$ is closed, we have that for any $\sigma\in\Gamma F$
\[ d\eta(X,Y)d_{Z}\sigma =0.\]
Since $\eta F^{(1)}=0$, this is equivalent to 
\[ \eta(X)d_{YZ}\sigma - \eta(Y) d_{XZ}\sigma=0.\]
Assume now that $X$ and $Y$ are distinct asymptotic directions of $F$. Then setting $Z=Y$ implies that $\eta(Y)d_{XY}\sigma = 0$, since $d_{YY}\sigma \in\Gamma F^{(1)}$. Similarly, setting $Z=X$ implies that  $\eta(X)d_{YX}\sigma = 0$, which in turn implies that $\eta(X)d_{XY}\sigma = 0$. Therefore, since $\{ \sigma , d_{X}\sigma , d_{Y}\sigma , d_{XY}\sigma\}$ is a basis for $\underline{\mathbb{R}}^{4}$, $\eta =0$. Thus we have proved the following classically known theorem:

\begin{theorem}
Surfaces in projective $3$-space are rigid to third order. 
\end{theorem}

\section{Hypersurfaces in the conformal $n$-sphere}
\label{sec:confdefo}
In this section we will apply the results of Section~\ref{sec:defo} to examine deformations of hypersurfaces in conformal geometry. For a modern treatment of conformal geometry see for example \cite{AG1996,AG1997,BCwip,BC2010i,H2003,M1995,JMN2016}.  

Let $n\in\mathbb{N}$. Then we may view $\mathbb{S}^{n}$ as the projective light cone $\mathbb{P}(\mathcal{L})$ of $\mathbb{R}^{n+1,1}$, which is acted upon transitively by the orthogonal group $\textrm{O}(n+1,1)$. Suppose that $F:\Sigma\to \mathbb{P}(\mathcal{L})$ is an immersion, where $\Sigma$ is an $(n-1)$-dimensional manifold. We will view $F$ as a null line subbundle of $\underline{\mathbb{R}}^{n+1,1}$. Note that as $F$ is an immersion, the derived bundle $F^{(1)}$ of $F$ is a codimension 1 subbundle of $F^{\perp}$. Let $V$ be a sphere congruence enveloped by $F$, i.e., $V$ is a bundle of $(n,1)$-planes such that $F^{(1)}\le V$. Then let $\tilde{F}$ be a null-line subbundle of $V$ complementary to $F$, i.e., $F\oplus \tilde{F}$ is a $(1,1)$-subbundle of $V$. Let $U:= (F\oplus \tilde{F})^{\perp}\cap V$. Then $F^{(1)} = F\oplus U$ and $F^{\perp}=F\oplus U \oplus V^{\perp}$. We now have a splitting
\[\underline{\mathbb{R}}^{n+1,1} = F \oplus \tilde{F} \oplus U\oplus V^{\perp},\]
and thus a splitting of $\wedge^{2}\underline{\mathbb{R}}^{n+1,1}$:
\[ \wedge^{2}\underline{\mathbb{R}}^{n+1,1} = F\wedge U\oplus F\wedge V^{\perp} \oplus U \wedge U \oplus U\wedge V^{\perp}  \oplus F\wedge \tilde{F} \oplus \tilde{F}\wedge U\oplus \tilde{F}\wedge V^{\perp}.\]
We shall make the assumption that $F$ is an umbilic-free hypersurface. This implies that the derived bundle $F^{(2)}$ of $F^{(1)}$ satisfies $F^{(2)}= \underline{\mathbb{R}}^{n+1,1}$. 

\subsection{Second order deformations}
By Theorem~\ref{thm:defozero}, $\eta\in \Omega^{1}(\underline{\mathfrak{o}(n+1,1)})$ is a second order infinitesimal deformation of $F$ if and only if $\eta$ satisfies the Maurer-Cartan equation, and for all $v_{0}\in (\mathbb{R}^{n+1,1})^{*}$ and $X,Y\in \Gamma T\Sigma$
\begin{eqnarray}
\label{eqn:confsord}
\eta \sigma = v_{0}(\eta \sigma)\sigma\quad
\text{and}\quad \eta(X)d_{Y}\sigma = v_{0}(\eta(X)\sigma)d_{Y}\sigma + v_{0}(\eta(X)d_{Y}\sigma)\sigma,
\end{eqnarray}
where $\sigma\in\Gamma F$ such that $v_{0}(\sigma)=1$. From the skew-symmetry of $\eta$ it follows that $v_{0}(\eta \sigma)=0$. Thus, (\ref{eqn:confsord}) holds if and only if 
\[ \eta F =0 \quad \text{and}\quad \eta F^{(1)}\le \Omega^{1}(F), \]
or equivalently
\[ \eta F= 0 \quad \text{and} \quad \eta \,U \le \Omega^{1}(F).\]
This clearly holds if and only if 
\[ \eta \in\Omega^{1}(F\wedge U \oplus F\wedge V^{\perp}) =\Omega^{1}(F\wedge F^{\perp}).\]
Now $F\wedge F^{\perp}$ is a bundle of abelian subalgebras of $\underline{\mathfrak{o}(n+1,1)}$. Therefore, $[\eta\wedge\eta]=0$ and the condition that $\eta$ satisfies the Maurer-Cartan equation reduces to $\eta$ being closed. 

We shall now investigate the uniqueness and triviality of second order deformations. According to Lemma~\ref{lem:defouniq} and Lemma~\ref{lem:defotriv}, this is determined by maps $h:\Sigma\to \textrm{O}(n+1,1)$ such that $F$ is a second order deformation of itself via $h$, i.e., $h$ such that $hF =F$ and $\theta_{h}:= h^{-1}dh\in\Omega^{1}(F\wedge F^{\perp})$. Thus, for any section $\sigma\in\Gamma F$, $h\sigma =\lambda \sigma$, for some smooth function $\lambda$. Differentiating this along $X\in\Gamma T\Sigma$ gives
\[ (d_{X}h)\sigma + h d_{X}\sigma = (d_{X}\lambda)\sigma + \lambda d_{X}\sigma.\]
But since $\theta_{h}F=0$, we have that 
\[ h d_{X}\sigma = (d_{X}\lambda)\sigma + \lambda d_{X}\sigma.\]
The orthogonality of $h$ then gives that $\lambda=\pm 1$. Hence, $h|_{F^{(1)}}= \pm id|_{F^{(1)}}$ and for any $\nu\in\Gamma F^{(1)}$, $h\nu=\pm\nu$. Differentiating this condition along $Y\in\Gamma T\Sigma$ gives that
\[ (d_{Y}h)\nu + hd_{Y}\nu = \pm d_{Y}\nu.\]
Then since $\theta_{h}F^{\perp}\le F$, we have that $h|_{F^{(2)}}\equiv \pm id|_{F^{(2)}}\, mod\, F$. Now using our assumption that $F^{(2)}=\underline{\mathbb{R}}^{n+1,1}$, we may write 
\[ h= \pm id + \xi,\]
where $\xi|_{F^{(1)}}=0$ and $im \xi\le F$. From the orthogonality of $h$ one may deduce that $\xi$ is skew-symmetric. Combined with $\xi|_{F^{(1)}}=0$ and $im \xi\le F$, this can only hold if $\xi =0$. We therefore have the following lemmata:

\begin{lemma}
\label{lem:confuniq}
Suppose that $\hat{F}$ is a second order deformation of $F$ via $g$. Then $\hat{F}$ is a second order deformation of $F$ via $\tilde{g}$ if and only if $\tilde{g} = \pm g$. 
\end{lemma}

\begin{lemma}
$\eta$ is a trivial second order infinitesimal deformation of $F$ if and only if $\eta=0$. 
\end{lemma}

We have thus arrived at the main theorem of this subsection:

\begin{theorem}
\label{thm:confsord}
$F:\Sigma\to \mathbb{P}(\mathcal{L})$ is deformable of order two if and only if there exists a closed non-zero one-form $\eta$ taking values in $F\wedge F^{\perp}$. 
\end{theorem}

In \cite{B2006} it is shown that an $\eta$ satisfying the conditions of Theorem~\ref{thm:confsord} does not exist for $n>3$. In the case of $n=3$, one sees that the conditions of Theorem~\ref{thm:confsord} coincide with the gauge-theoretic definition of \textit{isothermic surfaces} (see for example~\cite{BC2010i,BS2012}). Classically, these are the surfaces that, away from umbilic points, locally admit  conformal curvature line coordinates. We have thus recovered the following result:

\begin{corollary}[\cite{C1923}]
Isothermic surfaces are the only second order deformable surfaces in the conformal 3-sphere. 
\end{corollary}

\begin{remark}
In~\cite{C1920i,MN2009}, the deformability of submanifolds in the conformal $n$-sphere with codimension greater that one was considered. In this case it is shown that, although isothermic surfaces are deformable to second order, a generic second order deformable surface is not isothermic. 
\end{remark}

In \cite{S2008} it was proved that more can be said about where $\eta$ takes values:

\begin{proposition}
\label{prop:confsord}
If $\eta\in\Omega^{1}(F\wedge F^{\perp})$ is closed then $\eta\in \Omega^{1}(F\wedge F^{(1)})$.
\end{proposition}

\subsection{Third order deformations} 
We will now show that rigidity occurs at third order in the conformal $3$-sphere. Suppose that $\eta$ is a third order infinitesimal deformation of $F$. Then by Proposition~\ref{prop:confsord}, $\eta \in \Omega^{1}(F\wedge F^{(1)})$. Furthermore, by Theorem~\ref{thm:defo}, for all $X,Y,Z\in\Gamma T\Sigma$ and $\sigma\in \Gamma F$, 
\[ (d_{Y}\eta(Z))\sigma =\mu\, \sigma\quad \text{and}\quad (d_{X}d_{Y}\eta(Z))\sigma \in \Gamma F, \]
for some smooth function $\mu$. Using the Leibniz rule, one then deduces that 
\[ (d_{Y}\eta(Z))d_{X}\sigma = \mu \, d_{X}\sigma \, mod\, F .\] 
The skew-symmetry of $(d_{Y}\eta(Z))$ implies that $\mu=0$. Hence, $(d_{Y}\eta(Z))\sigma=0$. By the Leibniz rule this implies that $\eta(Z)d_{Y}\sigma =0$ and thus $\eta F^{(1)}=0$. Therefore, $\eta=0$ and it follows that:

\begin{theorem}
A surface in the conformal $3$-sphere is rigid to third order.
\end{theorem}

\section{Legendre maps}
\label{sec:legdefo}
In this section we study the deformability of contact elements in Lie sphere geometry and projective geometry. This problem has been studied in~\cite{B1929,F2000i,F2000ii,F1916,MN2006}. 

Let $s,t\in\mathbb{N}$ such that $(s,t)=(3,3)$ or $(s,t)=(4,2)$. Consider $\mathbb{R}^{s,t}$ and let $\mathcal{L}^{5}$ denote the $5$-dimensional lightcone of this space. Let $\mathcal{Z}$ denote the Grassmannian of null two dimensional subspaces of $\mathbb{R}^{s,t}$. $\mathcal{Z}$ is acted upon transitively by $G= \textrm{O}(s,t)$. Suppose that $\Sigma$ is a 2-dimensional manifold. We say that a smooth map $f:\Sigma\to \mathcal{Z}$ is a \textit{Legendre map} if $f^{(1)}\le f^{\perp}$ and at every $p\in \Sigma$, if $X\in T_{p}\Sigma$ such that $d_{X}\sigma \in f(p)$ for all sections $\sigma\in\Gamma f$, then $X=0$. We may view a Legendre map as rank $2$ null subbundle on the trivial bundle $\underline{\mathbb{R}}^{s,t}:=\Sigma\times \mathbb{R}^{s,t}$.

It was shown in~\cite{BH2002} that a Legendre map naturally equips $T\Sigma$ with a conformal structure. In the case that $(s,t)=(4,2)$ this conformal structure at each point either vanishes or has signature $(1,1)$, however in the case of $(s,t)=(3,3)$, any signature is possible. From this point onwards we shall make the assumption that the signature of this conformal structure is $(1,1)$ at each point. In this case we may denote by $T_{1}$ and $T_{2}$ the null subbundles of this conformal structure. Our Legendre map then admits two special rank $1$ subbundles $s_{1}$ and $s_{2}$, called the \textit{curvature sphere congruences of $f$}, such that
\[ d_{X}\sigma_{1}, d_{Y}\sigma_{2}\in \Gamma f,\]
for all $\sigma_{1}\in \Gamma s_{1}$, $\sigma_{2}\in \Gamma s_{2}$, $X\in \Gamma T_{1}$ and $Y\in \Gamma T_{2}$. We may then form a splitting of the trivial bundle $\underline{\mathbb{R}}^{s,t}$ as $\underline{\mathbb{R}}^{s,t}=S_{1}\oplus_{\perp} S_{2}$, where 
\begin{equation}
\label{eqn:liecyc}
S_{1} := \langle \sigma_{1}, d_{Y}\sigma_{1}, d_{Y}d_{Y}\sigma_{1}\rangle\quad \text{and}\quad S_{2}:=\langle \sigma_{2},d_{X}\sigma_{2}, d_{X}d_{X}\sigma_{2}\rangle.
\end{equation}
This is called the \textit{Lie cyclide splitting}. For $i\in \{1,2\}$, let $f_{i}$ denote the set of sections of $f$ and derivatives of $f$ along $T_{i}$. One then has that $f_{i}$ is a rank $3$ subbundle of $f^{\perp}$ and furthermore 
\[ f^{\perp}/f = f_{1}/f\oplus_{\perp}f_{2}/f,\]
with each $f_{i}/f$ inheriting a non-degenerate metric from that of $\underline{\mathbb{R}}^{s,t}$. 

We identify $f$ with the map $F:\Sigma\to Z$, defined by $F = \wedge^{2}f$, where $Z$ is the subset of $\mathbb{P}(\wedge^{2}\mathbb{R}^{s,t})$ defined by
\[ Z:= \{ [v\wedge w]: v,w\in\mathcal{L}, \, v\not\parallel w \,\, \text{and}\,\,  (v,w)=0 \}.\]
$Z$ is acted upon smoothly and transitively by $\textrm{O}(s,t)$ via 
\[ A [v\wedge w] = [Av\wedge Aw].\]

Let $\tilde{f}:\Sigma\to \mathcal{Z}$ be complementary to $f$, i.e., $f\oplus \tilde{f}$ is a rank $4$ bundle with signature $(2,2)$. Let $U=(f\oplus \tilde{f})^{\perp}$. Then we have a splitting of $\underline{\mathbb{R}}^{s,t}$:
\[ \underline{\mathbb{R}}^{s,t}= (f\oplus \tilde{f})^{\perp} \oplus_{\perp} U.\]
This induces a splitting of $\wedge^{2}\underline{\mathbb{R}}^{s,t}$: 
\[ \wedge^{2}\underline{\mathbb{R}}^{s,t} = \wedge^{2} f \oplus f\wedge U \oplus f\wedge \tilde{f} \oplus \wedge^{2} U \oplus \tilde{f}\wedge U \oplus \wedge^{2} \tilde{f}.\]

\subsection{Second order deformations}
\label{subsec:cfsord}
By Theorem~\ref{thm:defo}, $\eta\in \Omega^{1}(\underline{\mathfrak{o}(s,t)})$ is a second order infinitesimal deformation if and only if $\eta$ satisfies the Maurer-Cartan equation and 
\begin{eqnarray}
\label{eqn:sord}
 \eta F\le \Omega^{1}(F) \quad \text{and}\quad (d_{X}\eta(Y))F\le F,
\end{eqnarray}
for all $X,Y\in\Gamma T\Sigma$. Now $\eta F\le \Omega^{1}(F)$ if and only if for linearly independent $\sigma,\xi\in\Gamma f$, 
\[ (\eta \sigma)\wedge \xi + \sigma \wedge (\eta \xi) = \eta (\sigma\wedge \xi) \in \Omega^{1}(F).\]
Since $\sigma$ and $\xi$ are linearly independent this is equivalent to 
\[ \eta f\le \Omega^{1}(f).\]
Similarly, one can show that $(d_{X}\eta(Y))F\le F$ is equivalent to $(d_{X}\eta(Y))f\le f$. By the Leibniz rule, this holds if and only if for any section $\sigma\in \Gamma f$,
\begin{eqnarray}
\label{eqn:sord2}
d_{X}(\eta(Y)\sigma) - \eta(Y)d_{X}\sigma \in \Gamma f.
\end{eqnarray}
Now, if we assume that $X$ is a curvature direction, i.e., $X\in \Gamma T_{i}$ for some $i\in\{1,2\}$, then $\eta f\le \Omega^{1}(f)$ implies that $d_{X}(\eta(Y)\sigma)\in \Gamma f_{i}$. Furthermore, $\eta(Y)d_{X}\sigma$ is orthogonal to $d_{X}\sigma$. Therefore, as the metric on $\underline{\mathbb{R}}^{s,t}$ restricts to a non-degenerate metric on $f_{i}/f$, we can deduce that 
\[ d_{X}(\eta(Y)\sigma), \, \eta(Y)d_{X}\sigma \in \Gamma f.\]
Now, $d_{X}(\eta(Y)\sigma)\in\Gamma f$ if and only if $\eta(Y)\sigma\in \Gamma s_{i}$. Since this holds for all $i\in\{1,2\}$, one has that $\eta f \equiv 0$. Also, $\eta(Y)d_{X}\sigma \in \Gamma f$ implies that $\eta f^{(1)} \le \Omega^{1}(f)$. Thus, $\eta \, U \le \Omega^{1}(f)$. Finally, 
\[ \eta f \equiv 0 \quad \text{and}\quad \eta \, U \le \Omega^{1}(f)\]
if and only if 
\[ \eta \in \Omega^{1}(\wedge^{2} f\oplus f\wedge U) = \Omega^{1}(f\wedge f^{\perp}).\]
One can easily check that the converse is true, i.e., given $\eta\in\Omega^{1}(f\wedge f^{\perp})$ satisfying the Maurer-Cartan equation, (\ref{eqn:sord}) holds.

The following proposition was proved in~\cite{P2018} in the case that $(s,t)=(4,2)$. Using analogous arguments one can show that it holds in the case that $(s,t)=(3,3)$ as well. 

\begin{proposition}
\label{prop:closedform}
Suppose that $\eta\in \Omega^{1}(f\wedge f^{\perp})$. Then $\eta$ satisfies the Maurer-Cartan equation if and only if it is closed. Furthermore, $\eta(T_{i})\le f\wedge f_{i}$ and $[\eta\wedge \eta]=0$. 
\end{proposition}

Thus, we have arrived at the following proposition:

\begin{proposition}
\label{prop:cfsord}
$\eta\in \Omega^{1}(\underline{\mathfrak{o}(s,t)})$ is a second order infinitesimal deformation of $f$ if and only if $\eta$ is closed and takes values in $f\wedge f^{\perp}$. 
\end{proposition}

We now wish to determine the uniqueness and triviality of such deformations. Following Lemma~\ref{lem:defouniq} and Lemma~\ref{lem:defotriv}, we investigate maps $h:\Sigma\to \textrm{O}(s,t)$ such that $F$ is a second order deformation of itself via $h$. By Proposition~\ref{prop:cfsord}, such a $h$ is characterised by
\begin{equation}
\label{eqn:legtrivdef}
h F=F \quad \text{and} \quad \theta_{h}:=h^{-1}dh\in\Omega^{1}(f\wedge f^{\perp}).
\end{equation}
Furthermore, $hF=F$ if and only if $hf=f$. Let $\sigma_{i}\in\Gamma s_{i}$ be a lift of one of the curvature spheres of $f$. Then, since $hf=f$ we have that 
\[ h\sigma_{i} = \nu,\]
for some $\nu\in\Gamma f$. Differentiating this condition with respect to the curvature direction $X\in \Gamma T_{i}$ yields 
\[ (d_{X}h) \sigma_{i} + h d_{X}\sigma_{i} = d_{X}\nu.\]
Since $\theta_{h}\in\Omega^{1}(f\wedge f^{\perp})$, we have that $(d_{X}h) \sigma_{i} = 0$ and thus 
\[ h d_{X}\sigma_{i} = d_{X}\nu.\]
Since $d_{X}\sigma_{i}\in\Gamma f$ and $hf=f$, we must have that $d_{X}\nu\in \Gamma f$. Thus, $\nu \in\Gamma s_{i}$. Therefore, for some smooth function $\lambda$ we have that $h\sigma_{i} = \lambda \sigma_{i}$. Differentiating this condition gives for any $Z\in\Gamma T\Sigma$, 
\begin{eqnarray}
\label{eqn:sorduniq}
 (d_{Z}h)\sigma_{i} + h d_{Z}\sigma_{i} = (d_{Z}\lambda) \sigma_{i} + \lambda d_{Z}\sigma_{i}.
\end{eqnarray}
Then the orthogonality of $h$ and that $\theta_{h}f\equiv 0$ implies that $\lambda = \pm 1$. Therefore, $h|_{s_{i}}=\pm id|_{s_{i}}$. We then have two cases to consider either $h|_{f}=\pm id|_{f}$ or $h|_{s_{1}}=\pm id|_{s_{1}}$ and $h|_{s_{2}}=\mp id|_{s_{2}}$. 

\begin{lemma}
Suppose that $h|_{s_{1}} =\pm id|_{s_{1}}$ and $h|_{s_{2}}=\mp id|_{s_{2}}$. Then $S_{1}$ and $S_{2}$ are constant. 
\end{lemma}
\begin{proof}
Let $\sigma_{1}\in\Gamma s_{1}$ and $\sigma_{2}\in\Gamma s_{2}$ and let $X\in \Gamma T_{1}$ and $Y\in \Gamma T_{2}$. Then 
\[ d_{X}\sigma_{1} = \alpha_{1} \sigma_{1} + \beta_{1}\sigma_{2} \quad \text{and}\quad d_{Y}\sigma_{2} = \alpha_{2} \sigma_{1} + \beta_{2}\sigma_{2},\]
for smooth functions $\alpha_{1},\alpha_{2},\beta_{1},\beta_{2}$. Now 
\[ \pm(\alpha_{1} \sigma_{1} + \beta_{1}\sigma_{2} ) = \pm d_{X}\sigma_{1} = d_{X}(h\sigma_{1}) =
(d_{X}h)\sigma_{1} + h d_{X}\sigma_{1} = \pm \alpha_{1} \sigma_{1} \mp \beta_{1}\sigma_{2},\]
since $\theta_{h}f\equiv 0$. Thus $\beta_{1}=0$. Similarly, one can show that $\alpha_{2}=0$. Thus, $d_{X}\sigma_{1}\in \Gamma s_{1}$ and $d_{Y}\sigma_{2}\in \Gamma s_{2}$ and one deduces from~(\ref{eqn:liecyc}) that $S_{1}$ and $S_{2}$ are constant. 
\end{proof}

$S_{1}$ and $S_{2}$ can only be constant if $f$ is a Dupin cyclide. In that case we may define $\rho\in \textrm{O}(s,t)$ such that $\rho$ restricts to the identity on $S_{1}$ and minus the identity on $S_{2}$. One then has that $F$ is a second order deformation of itself via $\tilde{h}:= \rho h$ and $\tilde{h}|_{f}=\pm id|_{f}$. 

So let us now assume that $h|_{f} = \pm id|_{f}$. Then by~(\ref{eqn:sorduniq}), $h|_{f^{(1)}} = \pm id|_{f^{(1)}}$. By differentiating this condition again one finds that $h|_{f^{(2)}/f} =\pm id |_{f^{(2)}/f}$, where $f^{(2)}$ denotes the derived bundle of $f^{(1)}$. Noting that $f^{(2)} = \underline{\mathbb{R}}^{s,t}$, we may write 
\[ h = \pm(id + \xi),\]
where $\xi$ satisfies $\xi(\underline{\mathbb{R}}^{s,t}) \le f$ and $\xi f^{\perp}\equiv 0$. Since $\xi(\underline{\mathbb{R}}^{s,t}) \le f$, we have that $(\xi v,\xi w)=0$ for all $v,w\in \Gamma\underline{\mathbb{R}}^{s,t}$. The orthogonality of $h$ then implies that $\xi$ is skew-symmetric. Combining this with the fact that $\xi(\underline{\mathbb{R}}^{s,t}) \le f$ and $\xi f^{\perp}\equiv 0$ we have that $\xi\in \Gamma (\wedge^{2} f)$. Hence, $h=\pm\exp(\xi)$.

Conversely, it is straightforward to check that if $h = \pm \exp(\xi)$, for some $\xi \in\Gamma (\wedge^{2} f)$, then $h$ satisfies~(\ref{eqn:legtrivdef}). We have thus arrived at the following lemmata:

\begin{lemma}
Suppose that $\hat{f}$ is a second order deformation of $f$ via $g$. Then, in the case that $f$ is not a Dupin cyclide, $\hat{f}$ is a second order deformation of $f$ via $\tilde{g}$ if and only if $\tilde{g}=\pm g\exp(\xi)$, for some $\xi\in\Gamma (\wedge^{2} f)$. In the case that $f$ is a Dupin cyclide, $\hat{f}$ is a second order deformation of $f$ via $\tilde{g}$ if and only if either $\tilde{g}=\pm g\exp(\xi)$ or $\tilde{g}=\pm \rho g\exp(\xi)$, for some $\xi\in\Gamma (\wedge^{2} f)$. 
\end{lemma}

\begin{lemma}
\label{lem:legtriv}
$\eta$ is a trivial second order infinitesimal deformation of $f$ if and only if $\eta = d\xi$ for some $\xi\in \Gamma (\wedge^{2} f)$. 
\end{lemma}

As shown in~\cite{P2018}, since $\sigma \mapsto \eta(X)d_{Y}\sigma$ defines an endomorphism $f\to f$, there is a quadratic differential
\[ q(X,Y) = \text{tr}(\sigma\mapsto \eta(X)d_{Y}\sigma)\]
associated to closed one-forms taking values in $f\wedge f^{\perp}$. It turns out that we may use $q$ to determine the triviality of $\eta$: 

\begin{lemma}
\label{lem:qtriv}
$q=0$ if and only if $\eta = d\xi$ for some $\xi\in \Gamma (\wedge^{2}f)$. 
\end{lemma}
\begin{proof}
We may write an arbitrary closed one-form $\eta\in \Omega^{1}(f\wedge f^{\perp})$ as 
\[ \eta = \alpha\,\sigma_{1}\wedge d\sigma_{1} + \beta\,\sigma_{2}\wedge d\sigma_{1} + \gamma\, \sigma_{1}\wedge d\sigma_{2} + \delta\,\sigma_{2}\wedge d\sigma_{2} \, mod\, \Omega^{1}(\wedge^{2} f)\]
for $\sigma_{1}\in \Gamma s_{1}$, $\sigma_{2}\in \Gamma s_{2}$ and some smooth functions $\alpha, \beta, \gamma, \delta$. The quadratic differential of $\eta$ is then given by
\[ q = - \alpha (d\sigma_{1}, d\sigma_{1}) - \delta (d\sigma_{2},d\sigma_{2}).\]
Since $(d\sigma_{1}, d\sigma_{1}) \in \Gamma (T_{2}^{*})^{2}$ and $(d\sigma_{2}, d\sigma_{2}) \in \Gamma (T_{1}^{*})^{2}$, one has that $q=0$ if and only if $\alpha=\delta=0$. One can clearly see that if $\eta=d\xi$, for some $\xi:=\lambda \sigma_{1}\wedge \sigma_{2}$, then $\alpha=\delta=0$. On the other hand, if $\alpha=\delta=0$, then the closure of $\eta$ implies that $\beta = -\gamma$ and moreover $\eta = d(\beta \sigma_{2}\wedge \sigma_{1})$. Hence $\eta = d\xi$ for $\xi := \beta \sigma_{2}\wedge \sigma_{1}$. 
\end{proof}

We thus obtain the main theorem of this section:
\begin{theorem}
\label{thm:legapp}
$f:\Sigma\to \mathcal{Z}$ is deformable to second order if and only if there exists a closed one-form $\eta$ taking values in $f\wedge f^{\perp}$ such that $q\neq 0$. 
\end{theorem}

Considering the case of $(s,t)=(4,2)$ we can see that the conditions of Theorem~\ref{thm:legapp} coincide with the gauge theoretic definition of \textit{$\Omega$-surfaces} and \textit{$\Omega_{0}$-surfaces} of~\cite{P2018}. Classically~\cite{D1911iii,D1911i,D1911ii}, $\Omega_{0}$-surfaces are defined as those surfaces possessing a curvature sphere congruence that is isothermic, whereas $\Omega$-surfaces are defined by the existence of a pair of enveloped isothermic sphere congruences that separate the curvature sphere congruences harmonically. We thus arrive at the following result:

\begin{corollary}[\cite{MN2006}]
$\Omega$- and $\Omega_{0}$-surfaces are the only second order deformable surfaces of Lie sphere geometry. 
\end{corollary}

\begin{remark}
In~\cite{BHPR2019,MN2006} it was shown how second order deformable maps in Lie sphere geometry yield deformable maps in conformal and Laguerre geometry. For more information about deformability in Laguerre geometry, see~\cite{MN1997,MN2016}. 
\end{remark}

\subsection{Third order deformations}
In this subsection we shall show that rigidity occurs at third order for Legendre maps. Suppose that $\eta$ is a third order infinitesimal deformation of $F$. Then by Theorem~\ref{thm:legapp}, $\eta \in \Omega^{1}(f\wedge f^{\perp})$ and $\eta$ is closed. Now by Theorem~\ref{thm:defo}, for $X, Y, Z\in \Gamma T\Sigma$, 
\[ (d_{X}d_{Y}\eta(Z))F \le F.\]
or, equivalently, 
\begin{eqnarray}
\label{eqn:cftord}
(d_{X}d_{Y}\eta(Z))f\le f.
\end{eqnarray}
Let $\sigma \in \Gamma f$ and assume that $X$ is a curvature direction of $f$, i.e, $X\in \Gamma T_{i}$ for $i\in\{1,2\}$. Then by the Leibniz rule, equation~(\ref{eqn:cftord}) implies that 
\begin{eqnarray}
\label{eqn:cftord2}
d_{X}((d_{Y}\eta(Z))\sigma) - (d_{Y}\eta(Z))d_{X}\sigma\in\Gamma f.
\end{eqnarray}
Now since $(d_{Y}\eta(Z))\sigma\in\Gamma f$, we have that $d_{X}((d_{Y}\eta(Z))\sigma) \in \Gamma f_{i}$. Furthermore, as $d_{Y}\eta(Z)$ is skew-symmetric, $(d_{Y}\eta(Z))d_{X}\sigma$ is orthogonal to $d_{X}\sigma$. Thus, equation~(\ref{eqn:cftord2}) holds if and only if 
\[ d_{X}((d_{Y}\eta(Z))\sigma) \in \Gamma f\quad \text{and} \quad (d_{Y}\eta(Z))d_{X}\sigma\in\Gamma f.\]
Now $d_{X}((d_{Y}\eta(Z))\sigma) \in \Gamma f$ implies that 
\[ (d_{Y}\eta(Z))\sigma\in \Gamma s_{i}.\]
Since $i$ was arbitrary, we then have that $(d_{Y}\eta(Z))\sigma=0$. By the Leibniz rule this implies that 
\[ d_{Y}(\eta(Z)\sigma) - \eta(Z)d_{Y}\sigma =0,\] 
and since $\eta(Z)f=0$, we have that 
\[ \eta(Z)d_{Y}\sigma =0.\] 
Hence, $\eta f^{\perp} \equiv 0$ and thus $\eta \in\Omega^{1}(\wedge^{2} f)$. One can then check that $\eta$ being closed implies that $\eta \equiv 0$. We have thus arrived at the following result:

\begin{theorem}
Legendre maps are rigid to third order.
\end{theorem}

\section{Projective applicability revisited}
\label{sec:projrev}
It is well known that surfaces in projective space $F:\Sigma\to\mathbb{P}(\mathbb{R}^{4})$ can be represented by their contact lifts in $\mathbb{R}^{3,3}$:
\[ f = F\wedge F^{(1)}.\]
The derived bundle of this contact lift is 
\[ f^{(1)} = F^{(1)}\wedge F^{(1)} + F\wedge \underline{\mathbb{R}}^{4}.\]
Recall also that there is an isomorphism $\phi: \mathfrak{sl}(4)\to \mathfrak{o}(3,3)$, defined by 
\[ \phi(A)\,  (v\wedge w) = Av\wedge w + v\wedge Aw.\]
Since $\phi$ is constant, $\phi$ intertwines the trivial connections on $\underline{\mathfrak{sl}(4)}$ and $\underline{\mathfrak{o}(3,3)}$. Let $\Theta\le \underline{\mathfrak{sl}(4)}$ denote the subbundle of $\underline{\mathfrak{sl}(4)}$ such that $A\in\Gamma\Theta$ if and only if 
\[ A F =0 \quad \text{and}\quad AF^{(1)}\le F.\] 
Then $\phi$ yields an isomorphism between $\Theta$ and $f\wedge f^{\perp}$. Since $\phi$ is constant one has that closed 1-forms taking values in $\Theta$ are in one-to-one correspondence with closed one forms taking values in $f\wedge f^{\perp}$. Furthermore, if we let $\Psi$ denote the subbundle of $\Theta$ defined by $A\in \Gamma \Psi$ if and only if 
\[ A F^{(1)}=0 \quad \text{and} \quad imA\le F,\]
then $\phi$ yields an isomorphism between $\Psi$ and $\wedge^{2}f$. Thus, one deduces that the triviality of second order infinitesimal deformations is preserved by $\phi$. We have thus recovered the classical result of Fubini~\cite{F1916}:

\begin{theorem}
A surface in projective $3$-space is deformable of order two if and only if its contact lift is deformable of order two.
\end{theorem}

Applying Theorem~\ref{thm:legapp} in the case that $(s,t)=(3,3)$, one observes that second order deformability of the contact lift of a surface in projective 3-space coincides with the gauge theoretic definition of projective applicability given in~\cite{C2012i}. It is shown in~\cite{F1937} that projectively applicable surfaces come in two classes, \textit{$R$-surfaces} and \textit{$R_{0}$-surfaces}. $R_{0}$-surfaces are characterised by the isothermicity of one of their asymptotic line congruences, whereas $R$-surfaces are characterised by the existence of a pair of enveloped isothermic line congruences that separate the asymptotic line congruences harmonically. We can therefore make the following statement: 

\begin{corollary}[\cite{C1920ii,F1937}]
$R$- and $R_{0}$-surfaces are the only second order deformable surfaces of projective geometry. 
\end{corollary}

\textit{Acknowledgements.} This work is based on part of the author's doctoral thesis~\cite{P2015}. The author would like to thank Professor Francis Burstall, who encouraged him to study this topic and provided much needed guidance in the creation of this paper. The author gratefully  acknowledges the support of the FWF through the research project P28427-N35 ``Non-rigidity and symmetry breaking" and the MIUR grant ``Dipartimenti di Eccellenza'' 20182022, CUP: E11G18000350001, DISMA, Politecnico di Torino.
\bibliographystyle{plain}
\bibliography{bibliography2017}

\end{document}